\documentclass{amsart}
\usepackage{amssymb}
\usepackage{amsmath}
\usepackage{stmaryrd}
\usepackage{amssymb}
\usepackage{array,multirow,makecell}
\usepackage[utf8]{inputenc}
\usepackage{amsmath}
\usepackage{amsthm}
\usepackage{amssymb}
\usepackage{amsfonts}
\usepackage{enumitem}
\usepackage{setspace}
\usepackage[bookmarks=true,hyperindex, pdftex,colorlinks,citecolor=blue]{hyperref}
\usepackage{centernot} %for negations
\usepackage{marginnote} % add margin notes
\usepackage[left=3cm,right=3cm,top=2cm,bottom=2cm]{geometry}
\usepackage{bbm}
\usepackage[all]{xy}
\usepackage{tikz}
\usetikzlibrary{arrows}
\usetikzlibrary{shapes,decorations}
\usetikzlibrary{positioning}
\def\<{\langle}
\def\>{\rangle}

\theoremstyle{plain}
\newtheorem{thm}{Theorem}[section]
\newtheorem{lemma}[thm]{Lemma}
\newtheorem{prop}[thm]{Proposition}

\theoremstyle{definition}
\newtheorem{defn}[thm]{Definition}

\theoremstyle{remark}

\numberwithin{equation}{section}

\begin{document}
%%%%%%%%%%%%%%%%%%%%%%%%%%

\title{disjoint strong transitivity of composition operators}
\author{Noureddine Karim, Otmane Benchiheb and Mohamed Amouch}

\address{Noureddine Karim, Otmane Benchiheb and Mohamed Amouch,
 Chouaib Doukkali University.
Department of Mathematics, Faculty of science
Eljadida, Morocco}
\email{noureddinekarim1894@gmail.com}
\email{otmane.benchiheb@gmail.com}
\email{amouch.m@ucd.ac.ma}
\keywords{Composition operators, disjoint hypercyclic operators, strong disjoint hypercyclic operators.}
\subjclass[2010]{47A16, 46A99.}

\begin{abstract}
A Furstenberg family $\mathcal{F}$ is a collection of infinite subsets of the set of positive integers such that if $A\subset B$ and $A\in \mathcal{F}$, then $B\in \mathcal{F}$. For a Furstenberg family $\mathcal{F}$, finitely many operators $T_1,...,T_N$ acting on a common topological vector space $X$ are said to be
disjoint $\mathcal{F}$-transitive if for every non-empty open subsets
$U_0,...,U_N$ of $X$ the set $\{n\in \mathbb{N}:\  U_0 \cap T_1^{-n}(U_1)\cap...\cap T_N^{-n}(U_N)\neq\emptyset\}$
belongs to $\mathcal{F}$. In this paper, depending on the topological properties of $\Omega$, we characterize the disjoint $\mathcal{F}$-transitivity of $N\geq2$ composition operators $C_{\phi_1},\ldots,C_{\phi_N}$ acting on the space $H(\Omega)$ of holomorphic maps on a domain $\Omega\subset \mathbb{C}$ by establishing a necessary and sufficient condition in terms of their symbols $\phi_1,...,\phi_N$.
\end{abstract}
%-----------------------------------
\maketitle
\section{Introduction}
Throughout this paper, $\mathbb{C}$ will represent the complex plane, $\mathbb{C}^*$ the punctured plane $\mathbb{C}\backslash \{0\}$, and $\hat{\mathbb{C}} = \mathbb{C}\cup\{\infty\}$ will be the one-point compactification of $\mathbb{C}$. Moreover, $\Omega$ will stand for a domain contained in $\mathbb{C}$. In contrast, $\mathbb{D}$ will represent the open unit disk of $\mathbb{C}$, and $H(\Omega)$ will denote the space of holomorphic functions on $\Omega$ endowed with the topology of uniform convergence on compact subsets.

If $X$ is a separable Fr\'{e}chet space, we denote by $\mathcal{B}(X)$ the algebra of operators, i.e., linear and continuous self-maps, and by $\mathcal{U}(X)$ the set of non-empty open subsets of $X$.

The most distinguished notions in a linear dynamical system setting $(X, T)$, where $T\in \mathcal{B}(X)$, are the cyclicity and the hypercyclicity; this is because they are connected with the invariant subspace problem and the invariant subset problem, respectively.

A sequence $(T_n)_{n\in \mathbb{N}}$ of operators on $X$ is said to be \emph{universal} if there exists a vector $f\in X$ such that
$$ \mathcal{O}((T_n)_{n\in \mathbb{N}},f):=\{T_nf:\ n\in \mathbb{N}\}$$
is dense in $X$.

An operator $T$ acting on $X$ is said to be \emph{hypercyclic} if the sequence of iterations $(T^n)_{n\in \mathbb{N}}$ of $T$ is universal, in other words, there is $f\in X$ having a dense orbit; that is, there exists a vector $f\in X$ such that
$$ \mathcal{O}(T,f):=\{T^nf:\ n\in \mathbb{N}\}$$
is dense in $X$. This is equivalent to the fact that for each $V\in\mathcal{U}(X)$ the set
$$ N_T(f,V)=N(f,V):=\{n\geq 0 \mbox{ : }T^nf\in V\} $$
is non-empty, or equivalently an infinite set.
The vector $f$ itself is called a \emph{hypercyclic vector} for $T$.
We denote by $HC(T)$ the set of all hypercyclic vectors for $T$.

The operator $T$ is said to be \emph{cyclic} if there exists a vector $f\in X$ such that the subspace generated by
the orbit;
$$\textrm{span}(\mathcal{O}(T,f))=\{p(T)f:\ p\ \textrm{a polynomial}\},$$
is dense in $X$.
In this case, the vector $f$ is called a \emph{cyclic vector} for $T$.
The set of cyclic vectors is denoted by $C(T).$

A notion lying between cyclicity and hypercyclicity is that of supercyclicity.
An operator $T$ is said to be \emph{supercyclic} if there is a vector $f\in X$ whose projective orbit;
$$\mathbb{C}\cdot\mathcal{O}(f,T):=\{\lambda T^n f:\ n\in \mathbb{N}\ \textrm{and}\ \lambda\in \mathbb{C} \},$$
is dense in $X$.
The vector $f$ is called a \emph{supercyclic vector} for $T$. The set of all supercyclic vectors for $T$ is denoted by $SC(T)$.

Utilizing the Birkhoff's Transitivity Theorem, see \cite{birkhoff}, one can prove that $T$ is hypercyclic if and only if it is \emph{topologically transitive}; that is, provided for each pair $(U,V)\in \mathcal{U}(X)$, the set
$$N_T(U,V)=N(U,V):=\{ n\in \mathbb{N} \mbox{ : } T^n(U)\cap V\neq \emptyset \}$$
is an infinite set.

For more information about hypercyclic, cyclic, and supercyclic operators and their proprieties, see \cite{AB1,AB2,AB3,AB4,bayart,universal,Peris}.

Recently in \cite{strong}, a generalization of topological transitivity, called $\mathcal{F}$-transitivity, where $\mathcal{F}$ is a family, was investigated. Recall that a non-empty
collection of subsets of $\mathbb{N}$ is a Furstenberg family or just a family if it satisfies the following property:
\begin{equation*}
 \emptyset\notin \mathcal{F}\ \textrm{and if}\ A \in \mathcal{F}\ \textrm{and}\ A \subset B,\ \textrm{then}\ B \in \mathcal{F}.
\end{equation*}
We say that an operator $T$ acting on $X$ is a $\mathcal{F}$-\emph{transitive} (or $\mathcal{F}$-operator for
short), provided $\mathcal{N}_T\subset \mathcal{F}$, where
$$ \mathcal{N}_T:=\{A\subset \mathbb{N} \mbox{ : }N_T(U,V) \subset A \mbox{ for some } U\mbox{, } V\in \mathcal{U}(X)\}. $$
This means that $T$ is $\mathcal{F}$-transitive if and only if $N_T(U,V)\in \mathcal{F}$ for each
$U$, $V$ non-empty open subsets of $X$.

For some families, the $\mathcal{F}$-transitivity are well-known notions in the topological dynamic.
%%%%%%%%%%%%%%%%%%%%%%%%%%%%%%%%%%%%%%%%%%%%%%%%%%%%%%%%%%
For instance, if $\mathcal{F}$ is the family of infinite sets, then the $\mathcal{F}$-operators are precisely the operators that are topologically transitive.
Suppose that $\mathcal{F}$ is the family of cofinite sets, then the $\mathcal{F}$-operators are the operators that are mixing, see \cite{mixing2,mixing3}.
If $\mathcal{F}$ is the family of thick sets (those that contain arbitrarily long intervals), then the $\mathcal{F}$-operators are the operators that are weak mixing, see \cite{wmixing1,wmixing2,wmixing3}. When $\mathcal{F}$ is the family of synthetic sets (those that have
bounded gaps) then the $\mathcal{F}$-operators are those that are ergodic.

Bès, Menet, Peris, and Puig in \cite{strong} characterized the $\mathcal{F}$-transitivity of weighted shift operators, while Amouch and Karim studied in \cite{AK} the $\mathcal{F}$-transitivity of composition operators.

The notion of disjointness was introduced by Harry Furstenberg in $1967$ in his seminal
paper \cite{D4} for dynamical systems where the spaces are measurable spaces, and the applications are measure-preserving transformations, and also for homeomorphisms of compact spaces, generally to study fluid flow. After that, the notions of disjoint universality, disjoint hypercyclicity, and disjoint topological transitivity were introduced and studied independently for operators by Bès and Peris in \cite{D3}, and Bernal Gonz\'{a}lez in \cite{D1}.\\
Let $N\in\mathbb{N}$, and $X$, $Y_1$, $\dots$, $Y_N$ be topological vector spaces.
For $n\in\mathbb{N}$ and $j=1$, $2$,  $\dots$, $N$, let $T_{j,n}$ : $X\longrightarrow Y_j$ be continuous linear mapping.
We say that the sequences $(T_{1,n})_{n=1}^{\infty}$, $\dots$ $(T_{N,n})_{n=1}^{\infty}$ are Disjoint Universal if the sequence $[T_{1,n},\dots, T_{N,n}]_{n=1}^{\infty}$ : $X\longrightarrow Y_1\times \dots \times Y_N$ defined as
$$ [T_{1,n},\dots, T_{N,n}]x=(T_{1,n}x,\dots, T_{N,n}x), \ \ x\in X \ \mbox{ and } \ n\in\mathbb{N}$$
is universal.
If $(T_1,$ $\dots,$ $T_N)\in\mathcal{B}(X)^N$, then the $N$-tuple $(T_1,$ $\dots,$ $T_N)$ are called disjoint hypercyclic provided that the sequences $(T_1^n)_{n=1}^{\infty},$ $\dots,$ $(T_N^n)_{n=1}^{\infty}$ are
disjoint universal.
Obviously, $(T_1,$ $\dots,$ $T_N)$ are disjoint hypercyclic if and only if the operator $T_1\oplus\dots\oplus T_N$ admits a hypercyclic vector of the form $(x,\ldots, x)$, for some $x \in X$. We say that the sequences of operators $(T_{1,n})_{n=1}^\infty,\ldots,(T_{N,n})_{n=1}^\infty$ in $\mathcal{B}(X)$
are $d$-topologically transitive provided for every non-empty open subsets
$U_0,...,U_N$ of $X$ there exists $n \in \mathbb{N}$ such that
$$U_0\cap T_{1,n}^{-1}(U_1)\cap \ldots\cap T_{N,n}^{-1}(U_N)\neq\emptyset.$$
 Also, we say that $N \geq 2$ operators
$T_1,...,T_N$ in $\mathcal{B}(X)$ are d-topologically transitive provided $(T_1^n)_{n=1}^\infty, . . . ,
(T_N^n)_{n=1}^{\infty}$ are $d$-topologically transitive sequences.

In $2009$, Salas in \cite{D2} introduced a notion that is stronger than the disjoint hypercyclicity with the name of dual disjoint hypercyclic operators.
Let $X$ be a Banach space and $(T_1,$ $\dots,$ $T_N)\in\mathcal{B}(X)^N$, then $(T_1,$ $\dots,$ $T_N)$ are said to be dual disjoint hypercyclic operators if both $(T_1,$ $\dots,$ $T_N)$ and $(T_1^*,$ $\dots,$ $T_N^*)$ are disjoint hypercyclic.

For inserting examples of disjoint hypercyclic $N$-tuples, see \cite{D1,D3,dishypweipseshionbanseqspa}.

One can ask, under which conditions a Banach space $X$ can support disjoint hypercyclic and dual disjoint hypercyclic operators.
Bés, Martin and Peris \cite{D5} and Salas \cite{D2} independently proved that if $N\in\mathbb{N}$ and $X$ is a infinite dimensional Banach space, then there exist $T_1$, $\dots$, $T_N\in\mathcal{B}(X)^N$ such that the $N$-tuple $(T_1,\dots, T_N)$ is disjoint hypercyclic.
Moreover, the construction in \cite{D5} still valid also for separable infinite-dimensional
Fréchet spaces, while the construction in \cite{D2} provides a dual disjoint hypercyclic $N$-tuple on an infinite-dimensional Banach space with separable dual.
%Later, Shkarin, in \cite{D6}, gave a short proof of those results by proving the existence of disjoint hypercyclic  of operators (disjoint dual hypercyclic tuples of operators) of any given
%length on any separable infinite-dimensional Fréchet space (infinite-dimensional
%Banach space with separable dual).
Later, Shkarin gave a short proof of those results in \cite{D6} by demonstrating the presence of disjoint hypercyclic operators (disjoint dual hypercyclic tuples of operators) of any length on any separable infinite-dimensional Fréchet space (infinite-dimensional Banach space with separable dual).

Recently it was introduced and studied the notion of disjoint $\mathcal{F}$-hypercyclicity and disjoint $\mathcal{F}$-topologically transitive of binary relations over topological spaces \cite{fhypdisfhyp,fhypextdisfhyp}.

A finitely many operators $T_1,\dots,T_N$ are said to be disjoint $\mathcal{F}$-transitive if for every $U_0,\dots,U_N\in \mathcal{U}(X)$, the set
$$N(U_0,...,U_N):=\{n\in \mathbb{N}:\ U_0\cap T_1^{-n} U_1\cap\dots\cap T_N^{-n} U_N\neq\emptyset\}$$ belongs to $\mathcal{F}$.

In this paper, we study the disjoint $\mathcal{F}$-transitivity of composition operators on $H(\Omega)$.
Recall that if $\phi$ is a holomorphic self-map of $\Omega$ then
the composition operator with symbol $\phi$ is defined as
$$C_\phi(f) = f \circ \phi\mbox{, } \ \ f \in H(\Omega).$$ Obviously we have that $$C_\phi^n=C_{\phi_n},\ \textrm{for every}\ n\in \mathbb{N},$$ where $\phi_n$ the $n$-th iterate of $\phi$,
$$\phi_n:=\phi\circ\dots\circ\phi,\ n-\textrm{times}.$$ From now on, if $\phi_1,\dots,\phi_N$ are $N$ self-maps of $\Omega$ we use the notation $\phi_{i,n}$ to denote the $n$-th iterate of $\phi_i$, $i=1,\dots,N$,
$$\phi_{i,n}:=\phi_i\circ\dots\circ\phi_i,\ n-\textrm{times}.$$

The first one who dealt with the dynamics of composition operator was
Birkhoff in \cite{birkhoff}, since then the dynamics of composition operators was studied by many authors, as examples
Bernal and Montes \cite{C1} studied the universality of a sequence of composition operators induced by automorphisms on the disk, Grosse-Erdmann and Mortini \cite{C2} (see also \cite{C3}) who characterized the hypercyclic composition generated by non-automorphisms self maps, Bayart, Darji, and Peris \cite{C4}, Bès \cite{C5,C6}, and Kamali and Yousefi \cite{dishypwei} and others more.

The paper is organized as follows: In Section $2$, we characterize the disjoint transitivity of the composition operator acting on $H(\Omega)$, where $\Omega$ is a finitely connected not simply connected, or infinitely connected domain. In this sense, we complete the study made by Bès and Martin \cite{comdishypequdissup}, who characterized the disjoint transitivity of the composition operator acting on $H(\Omega)$, where $\Omega$ is simply connected.\\
In Section $3$, we study the disjoint $\mathcal{F}$-transitivity on $H(\Omega)$, when $\Omega$ is a simply connected domain.
We prove that if $\phi_1,\dots,\phi_N$ be $N\geq2$ holomorphic self-maps of a simply connected domain $\Omega$, then  the operators $C_{\phi_1},\dots,C_{\phi_N}$ are disjoint $\mathcal{F}$-transitive
if and only if the maps $\phi_1,\dots,\phi_N$ are injective and
for every compact subset $K$ of $\Omega$ the set $$\{n\in \mathbb{N}:\ K,\phi_{1,n}(K),\dots,\phi_{N,n}(K)\textrm{ are pairwise disjoint}\}$$
belongs to $\mathcal{F}$.\\
In Section $4$, we turn into the case $\Omega$ is finitely connected but not simply connected or infinitely connected.
In particular, if $\Omega$ is finitely connected but not simply connected then there is no sequence $\phi_1,\dots,\phi_{N}$ of self-map of $\Omega$ such that $C_{\phi_1},\dots,C_{\phi_N}$ are disjoint $\mathcal{F}$-transitive. In case $\Omega$ is infinitely connected we show that if $\phi_1,\dots,\phi_{N}$ are $N$ holomorphic self-map of $\Omega$, then $C_{\phi_1},\dots,C_{\phi_N}$ are disjoint $\mathcal{F}$-transitive
if and only if $\phi_1,\dots\phi_N$ are injective and for every $\Omega$-convex compact subset $K$ of $\Omega$, the set
$$\{n \in \mathbb{N}\mbox{ : } \phi_{i,n}(K),\ i=1,\dots,N \mbox{ are } \Omega\mbox{-convex and } K,\phi_{1,n}(K),\dots,\phi_{N,n}(K) \mbox{ are pairwise disjoint}\}$$
belongs to $\mathcal{F}$.

\section{Disjoint transitivity on $H(\Omega)$}
In this section, we study the disjoint transitivity of $C_\phi$ on $H(\Omega)$, where $\Omega$ is any domain of $\mathbb{C}$ which is not simply connected. There are two cases:
if $\Omega$ is finitely connected non-simply connected, and the case in which $\Omega$ is infinitely connected. Since the hypercyclicity of composition operator $C_\phi$ implies the injectivity of $\phi$, and the disjoint transitivity is stronger than hypercyclicity, then from now on, we shall assume that the maps $\phi_1,\dots,\phi_N$ are injective.
\begin{defn}
A domain $\Omega$ of $\mathbb{C}$ is called simply connected if $\mathbb{\hat{C}}\backslash\Omega$ is connected. The domain $\Omega$ is called finitely
connected if $\hat{\mathbb{C}}\backslash \Omega$ contains at most finitely many connected components; otherwise, it is infinitely connected.
\end{defn}
For simplicity, if $M$ is a set of $\mathbb{C}$, any bounded component of $\mathbb{\hat{C}}\backslash M$ will be referred to as a hole of $M$. In this sense, finitely connected domains have a limited number of holes, a simply connected domain has none.
\begin{defn}
Let $\Omega$ be a domain of $\mathbb{C}$ and $M$ a compact subset of $\Omega$. The compact $M$ is said to be $\Omega$-convex if any hole of $M$ includes a point of $\mathbb{C}\backslash\Omega$.
\end{defn}
\begin{thm}
If $\Omega\subset \mathbb{C}$ is a finitely connected domain that is not simply connected. Then there is no family $C_{\phi_{1}},\dots ,C_{\phi_{N}}$ of disjoint transitive composition operators, where
$\phi_1,\dots,\phi_N$ are holomorphic injective self-maps.
\end{thm}
\begin{proof}
Assume that $C_{\phi_{1}},\dots ,C_{\phi_{N}}$ are disjoint transitive.
Then the operator $C_{\phi_{1}}\bigoplus\dots\bigoplus C_{\phi_{N}}$ is hypercyclic.
This implies that each operator $C_{\phi_{i}}$, $i=1,\dots, N$, is hypercyclic in $H(\Omega)$,
which is a contradiction since if
$\Omega$ is finitely connected not simply connected, then there is no hypercyclic composition operator on $H(\Omega)$, see \cite[Theorem 3.15]{C2}.
%Since disjoint transitivity of $C_{\phi_{1}},\dots ,C_{\phi_{N}}$ implies hypercyclicity of $C_{\phi_{1}}\bigoplus\dots\bigoplus C_{\phi_{N}}$ which implies the hypercyclicity of every $C_{\phi_i}$, $i=1,\dots,N$, but if $\Omega$ is finitely connected not simply connected there is no hypercyclic composition operator on $H(\Omega)$ by \cite[Theorem 3.15]{C2}.
\end{proof}
\begin{thm}\label{thm11}
Let $\Omega\subset \mathbb{C}$ be an infinitely connected domain, and let $\phi_{1},\dots,\phi_{N}$ be $N \geq 2$ holomorphic injective self-maps of $\Omega$. If the operators $C_{\phi_{1}},\dots ,C_{\phi_{N}}$ are d-topologically transitive, then for every $\Omega$-convex
compact subset $K$ of $\Omega$, there is some $n \in \mathbb{N}$ such that the sets
$$\phi_{1,n}(K),\phi_{2,n}(K),\dots,\phi_{N,n}(K)$$
are $\Omega$-convex and $$K,\phi_{1,n}(K),\dots,\phi_{N,n}(K)$$ are pairwise disjoint.
\end{thm}
\begin{proof}
Let $K$ be a $\Omega$-convex compact subset of $\Omega$. As in the proof of \cite[Theorem 3.12]{C2} we can construct
a compact $L$ that includes $K$ such that:
\begin{enumerate}
\item $L$ is $\Omega$-convex with a finite number of holes $O_1,\ldots,O_p$ whose
boundaries represent Jordan curves $\gamma_1,\ldots,\gamma_p$;
\item $L$ is connected and its outer boundary $\gamma_0$ embraces all the $p$ holes.
\end{enumerate}
If $b\in \Omega\backslash K$, then in each hole $O_j$ of $L$, $j = 1, . . . , p$, we can choose a point $a_j \in \mathbb{C}\backslash\Omega$.

%Then our curves $\gamma_j$ will have the following properties:
%$$\textrm{ind}_{\gamma_j}(a_j)=-1,\ \textrm{ind}_{\gamma_j}(a_k)=0,\ \textrm{ind}_{\gamma_j}(b)=0,$$
%$$\textrm{ind}_{\gamma_0}(a_j)=1,\ \textrm{ind}_{\gamma_0}(b)=1,\ j=1,\dots,p,$$
%where $\textrm{ind}_{\Gamma}(z)$ denotes the index of a point $z \in \mathbb{C}$ for the closed curve $\Gamma$.

For any $m \in \mathbb{N}$ we define
$$g_m(z)=m\frac{(z-b)^{p+1}}{\prod_{j=1}^p(z-a_j)}.$$
We then have:
$$\frac{1}{2i\pi}\int_{\gamma_j}\frac{g'_m(z)}{g_m(z)}dz= \frac{1}{2i\pi}\int_{\gamma_0}\frac{g'_m(z)}{g_m(z)}dz=1,$$
for $j = 1, . . . , p$ and $m \in \mathbb{N}$.
Let $f\in H(\Omega)$, such that $(f,f,\dots,f)$ be a hypercyclic vector for $C_{\phi_{1}}\oplus\dots\oplus C_{\phi_{N}}$.
Then, if $m\in \mathbb{N}$, there exist sequences $(n_k^{(m)})$ such that
$$f\circ \phi_{1,n_k^{(m)}}\rightarrow g_m, \ \ f\circ \phi_{2,n_k^{(m)}}\rightarrow g_m,\dots,f\circ \phi_{N,n_k^{(m)}}\rightarrow g_m,$$
 and
 $$(f\circ \phi_{1,n_k^{(m)}})'\rightarrow g'_m,\ \ (f\circ \phi_{2,n_k^{(m)}})'\rightarrow g'_m, \dots,(f\circ \phi_{N,n_k^{(m)}})'\rightarrow g'_m.$$
 Thus,
$\frac{(f\circ \phi_{i,n_k^{(m)}})'}{f\circ \phi_{i,n_k^{(m)}}}$ converge to $\frac{g'_m}{g_m}$ locally uniform on $\Omega\backslash\{b\}$, for every $i=1,\dots,N$. Since $\min_{z\in K}|g_m(z)|\rightarrow \infty$ as $m\rightarrow \infty$, taking into account that $b\notin K$, we deduce that there is a sequence $(n_m)$ such that, for every $i=1,\dots,N$
$$f(\phi_{i,n_m}(z))-g_m(z)\rightarrow 0\ \textrm{in}\ H(\Omega),$$
\begin{equation}\label{equ21}
\frac{(f\circ \phi_{i,n_m})'(z)}{(f\circ \phi_{i,n_m})(z)}-\frac{g'_m(z)}{g_m(z)}\rightarrow 0\ \textrm{in}\ H(\Omega\backslash\{b\}),
\end{equation}
and
$$\min_{z\in K}|f(\phi_{i,n_m}(z))|>\max_{z\in K}|f(z)|.$$
Thus, for every $i=1,\dots,N$, $$\phi_{i,n_m}(K)\cap K=\emptyset.$$
Suppose that, $\phi_{r,n_m}(K)\cap \phi_{s,n_m}(K)\neq\emptyset$ for some $r\neq s$. Choose $z_r,z_s\in K$ with $\phi_{r,n_m}(z_r)=\phi_{s,n_m}(z_s)$. Let $m\geq|r-s|$ such that there is $n_m$ with $\sup_{K}|f\circ\phi_{i,n_m}(z)-g_m(z)|<\frac{1}{2 m}$. By picking a particular $b$ we suppose that
$$\left\lvert\prod_{i=1}^{p}\frac{(z_r-b)^{p+1}}{(z_r-a_i)}-
\prod_{i=1}^{p}\frac{(z_s-b)^{p+1}}{(z_s-a_i)}\right\rvert\geq1.$$
 Then
\begin{align*}
1\leq  |r-s|&\leq m \left\rvert\prod_{i=1}^{p}\frac{(z_r-b)^{p+1}}{(z_r-a_i)}-
\prod_{i=1}^{p}\frac{(z_s-b)^{p+1}}{(z_s-a_i)}\right\rvert\\
                             &=|g_m(z_r)-g_m(z_s)|\\
                                                          &\leq |g_m(z_r)-f(\phi_{r,n_m}(z_r))|+|f(\phi_{s,n_m}(z_s))-g_m(z_s)|\\
                              &< \frac{1}{m},
\end{align*}
% $$1\leq  |p-q|\leq m \left\rvert\prod_{i=1}^{p}\frac{(z_p-b)^{p+1}}{(z_p-a_i)}-
%\prod_{i=1}^{p}\frac{(z_q-b)^{p+1}}{(z_q-a_i)}\right\rvert$$
%$$=|g_m(z_p)-g_m(z_q)|\leq |g_m(z_p)-f(\phi_{p,n_m}(z_p))|+|f(\phi_{q,n_m}(z_q))-g_m(z_q)|<\frac{1}{m},$$
which is contradicting.

From Equation \eqref{equ21}, there is some $n\in (n_m)_{m\in \mathbb{N}}$, such that $\phi_{1,n},\dots,\phi_{N,n}$ are injective on a neighborhood of $L$ and such that, for $q=0,\dots,p,$ and $i=1,\dots,N,$ we have that:
\begin{equation}\label{equ11}
\frac{1}{2\pi i}\int_{\phi_{i,n}(\gamma_q)}\frac{f'(z)}{f(z)}dz=\frac{1}{2\pi i}\int_{\gamma_q}\frac{(f\circ \phi_{i,n})'(z)}{(f\circ \phi_{i,n})(z)}dz=1.
\end{equation}

Now, for $i\in  \llbracket 1,N \rrbracket$, to show that $\phi_{i,n}(K)$ is $\Omega$-convex we will prove that $\phi_{i,n}(L)$ is $\Omega$-convex. This together with \cite[Lemma 3.11]{C2} implies that $\phi_{i,n}(K)$ is $\Omega$-convex.
Since $\phi_{i,n}$ is injective on a neighborhood of $L$, $\phi_{i,n}(L)$ is a compact set with exactly
$p$ holes (see \cite[p. 276]{remmert}). We assume that one of these holes, call it $O$, does not contain a point from
$\mathbb{C} \backslash \Omega$. Since injective holomorphic functions maps boundaries to boundaries, there is some
$l \in \{0, 1, . . . , p\}$ such that the Jordan curve $\phi_{i,n}(\gamma_l)$ is the boundary of $O$. Moreover, since
$O$ contains no point from $\mathbb{C}\backslash\Omega$, we have that
$$\textrm{ind}_{\phi_{i,n}(\gamma_l)}(\zeta) = 0\ \textrm{for}\ \zeta\notin \Omega.$$
Now, the compact set $L$ is to the left of each curve $\gamma_j$, $j = 0, 1, . . . , p$. Since injective
holomorphic mappings preserve orientation, $\phi_{i,n}(L)$ must also be to the left of the image
curve $\phi_{i,n}(\gamma_l)$. This implies that $\phi_{i,n}(\gamma_l)$ is oriented negatively. Since $f$ is holomorphic in a
neighborhood of $O$, the integral
$$-\frac{1}{2\pi i}\int_{\phi_{i,n}(\gamma_l)}\frac{f'(z)}{f(z)}dz$$
equals the number of zeros of $f$ in $O$; but, by Equation \eqref{equ11}, that integral has the value $-1$, a
contradiction. Hence, we can conclude that $\phi_{i,n}(L)$ is $\Omega$-convex.
\end{proof}
%In order to use the Runge approximation theorem, we must first recall the notion of an eventually injective sequence.
%\begin{defn}
%Let $\Omega\subset \mathbb{C}$ be a domain, and let $(\phi_n)$ be holomorphic self-maps of $\Omega$. Then
%$(\phi_n)$ is called eventually injective if, for every compact subset $K$ of $\Omega$, there is some $N \in \mathbb{N}$ such that $\phi_n|_{K}$ is injective for all $n \geq N$.
%\end{defn}
%\begin{rem}
%A holomorphic self-map $\phi$ is injective is equivalent to that the sequence of its iterates $(\phi_n)$ is eventually injective.
%\end{rem}
\begin{lemma}\label{lem12}
Let $\Omega\subset \mathbb{C}$ be an infinitely connected domain, and $\phi_1,\dots,\phi_N$ be $N\geq2$ injective holomorphic self-maps of $\Omega$. Suppose that, for every $\Omega$-convex compact subset $K$ of $\Omega$, there is some
$n\in \mathbb{N}$ such that for $i=1,\dots,N$, $\phi_{i,n}(K)$ are $\Omega$-convex and $K,\phi_{1,n}(K),\dots,\phi_{N,n}(K)$ are pairwise disjoint.
Then, for every connected $\Omega$-convex compact subset $K$ of $\Omega$ that has at least two holes, there is some $n \in \mathbb{N}$ such that
\begin{enumerate}
\item For every $i=1,\dots,N$, $\phi_{i,n}(K)$ is $\Omega$-convex and $K\cup \phi_{1,n}(K)\cup\dots\cup \phi_{N,n}(K)$ is $\Omega$-convex;
\item the sets $K,\phi_{1,n}(K),\dots,\phi_{N,n}(K)$ are pairwise disjoint.
\end{enumerate}
\end{lemma}
\begin{proof}
Let $K$ be a connected $\Omega$-convex compact subset of $\Omega$ with at least two holes. We fix an exhaustive sequence $(K_l)$ of $\Omega$ of connected $\Omega$-convex compact sets, all containing
$K$. Then, by hypothesis, there is a subsequence $(n_l)\subset \mathbb{N}$ such
that, for all $l \in \mathbb{N}$ and $i=1,\dots,N$, $\phi_{i,n_l}|_{K_{l}}$ is injective, $\phi_{i,n_l}(K_l)$ is $\Omega$-convex and $K_l,\phi_{1,n_l}(K_l),\dots,\phi_{N,n_l}(K_l)$ are pairwise disjoint.

Hence $K,\phi_{1,n_l}(K),\dots,\phi_{N,n_l}(K)$ are pairwise disjoint and, by \cite[Lemma 3.11]{C2}, for $i=1,\dots,N$ $\phi_{i,n_l}(K)$ is $\Omega$-convex, too. Put $$K'_l:=\phi_{1,n_l}(K)\cup\dots \cup \phi_{N,n_l}(K).$$ We want to prove that, for some $l \in \mathbb{N}$, $K\cup K'_l=K\cup \phi_{1,n_l}(K)\cup\dots\cup \phi_{N,n_l}(K)$ is $\Omega$-convex.

Three cases to distinguish:
\begin{enumerate}
\item[(a)] First, if, for some $l \in \mathbb{N}$, $K'_l$ lies in the unbounded
component of $\mathbb{C}\backslash K$ and $K$ lies in the unbounded component of $\mathbb{C} \backslash K'_l$ then it follows
immediately that $K'_l \cup K$ is $\Omega$-convex.

\item[(b)] Secondly, infinitely many of the $K'_l$ could be found in $K$ holes. Because
$K$ has finitely many holes, infinitely many $K'_l$ must reside in some fixed hole $O$ of
$K$ according to \cite[Lemma 3.10]{C2}. We can presume that all of them do by passing to a subsequence. Then we choose some $l \in \mathbb{N}$ such that for every $i=1,\dots,N$, $\phi_{i,n_1}(K) \subset K_l$. Since $\phi_{i,n_l}(K_l) \cap K_l=\emptyset$, we have that for every $i,j=1,\dots,N$,
$$\phi_{i,n_1}(K)\cap\phi_{j,n_l}(K)\subset K_l\cap\phi_{j,n_l}(K_l)=\emptyset.$$ Thus, $K'_l$ and $K'_1$ are disjoint subsets of $O$. There are currently three possibilities: If both of these sets lie in the
unbounded component of the complement of the other, $K'_l \cup K$ is $\Omega$-convex (as
is $K'_1 \cup K$); if $K'_1$ lies in a hole of $K'_l$ then $K'_1 \cup K$ is $\Omega$-convex because
$K'_l$ has at least two holes; or if $K'_l$ lies in a hole of $K'_1$ then $K'_l \cup K$ is
$\Omega$-convex since $K'_1$ has at least two holes.

\item[(c)] Finally, for infinitely many $l \in \mathbb{N}$, $K$ could be found in holes of $K'_l$. Again we can assume
this is true for all $l$. Then we can find some $l \in \mathbb{N}$ such that for each for every $i=1,\dots,N$, $\phi_{i,n_1}(K) \subset K_l$. Since $\phi_{i,n_l}(K_l) \cap K_l = \emptyset$ we know that $K'_1$ and $K'_l$ are disjoint sets. Because both these sets
contain $K$ in one of their holes, we must have that either $K'_1$ lies in a hole of $K'_l$
or $K'_l$ lies in a hole of $K'_1$. Then, as before, we argue that either $K'_l \cup K$ or
$K'_1 \cup K$ is $\Omega$-convex.
\end{enumerate}
\end{proof}
\begin{thm}
Let $\Omega\subset \mathbb{C}$ be a domain of infinite connectivity, and $\phi_1,\dots,\phi_N$ be injective holomorphic self-maps of $\Omega$. Then the following assertions
are equivalent:
\begin{enumerate}
\item The $N\geq2$ sequences $C_{\phi_1},\dots,C_{\phi_N}$ are disjoint topologically transitive;
\item For every $\Omega$-convex compact subset $K$ of $\Omega$, there is some $n\in \mathbb{N}$
such that the sets
 $$\phi_{1,n}(K),\dots,\phi_{N,n}(K)$$
  are $\Omega$-convex and
 $$K,\phi_{1,n}(K),\dots,\phi_{N,n}(K)$$
  are pairwise disjoint.
\end{enumerate}
\end{thm}
\begin{proof}
The necessity was shown in Theorem \ref{thm11}. We now show sufficiency.
Assume that $(2)$ holds. It suffices to show that $C_{\phi_1},\dots,C_{\phi_N}$ are d-topologically transitive. Let $U_0,\dots,U_N$ be non-empty open subsets of $H(\mathbb{D})$. Then there exist $\varepsilon > 0$, a compact $K \subset \mathbb{D}$, and functions $g_0,\dots,g_N \in H(\mathbb{D})$ such that
$$U_0 \supset \{h \in H(\mathbb{D}):\ \|h - g_0\|_K < \varepsilon\}$$
$$U_1 \supset \{h \in H(\mathbb{D}):\ \|h - g_1\|_K < \varepsilon\}$$
$$\vdots$$
and
$$U_N \supset \{h \in H(\mathbb{D}):\ \|h - g_N\|_K < \varepsilon\}.$$
By making $K$ larger, if needed, we can assume that it is connected, $\Omega$-convex and has at
least two holes. According to Lemma \ref{lem12} and the hypothesis there is some $n \in \mathbb{N}$ such that $\phi_{1,n},\dots,\phi_{N,n}$ are injective on a neighborhood of $K$,
$$K\cup\phi_{1,n}(K)\cup \dots\cup\phi_{N,n}(K)$$
is $\Omega$-convex, and $K,\phi_{1,n}(K),\dots,\phi_{N,n}(K)$ are pairwise disjoint. As a result, the functions $g_i\circ \phi_{i,n}^{-1}$  are holomorphic on $\phi_{i,n}(K)$ respectively for $i=1,\dots,N$, as well as $g_0$ on $K$. According to Runge's theorem (see, e.g., \cite[Ch. 13]{rudin}) there exists a function $h \in H(\Omega)$ such that
$$|h(z) - g_0(z)| < \varepsilon, \ \ \ \textrm{for}\ z \in K,$$
and
$$|h(w) - g_i(\phi_{i,n}^{-1}(w))| < \varepsilon, \ \ \ \textrm{for}\ w \in \phi_{i,n}(K).$$
Hence,
$$|h(\phi_{i,n}(z)) - g_i(z)| < \varepsilon, \ \ \ \textrm{for}\ z \in K.$$
Thus, $h\in U_0\cap C_{\phi_1}^{-n}U_1\cap\dots\cap C_{\phi_N}^{-n}U_N$, and so
$$U_0\cap C_{\phi_1}^{-n}U_1\cap\dots\cap C_{\phi_N}^{-n}U_N\neq\emptyset,$$
which implies that $C_{\phi_1},\dots,C_{\phi_N}$ are disjoint topologically transitive.
\end{proof}
\section{Disjoint $\mathcal{F}$-transitivity on $H(\Omega)$, when $\Omega$ is a simply connected domain}
We assume in this section that $\Omega$ is a simply connected domain of $\mathbb{C}$ with $\Omega\neq \mathbb{C}$. Then it is conformally equivalent to the disk $\mathbb{D}$. As a result, we can and shall assume that $\Omega$ is the disk $\mathbb{D}$ and our results still be valid.
\begin{defn}
Let $\phi_1,\dots,\phi_N$ be $N \geq 2$ holomorphic self-maps of a domain $\Omega$.
Then $\phi_1,\dots,\phi_N$ are said to be disjoint run-away if for every compact $K\subset \Omega$ there is an integer $n$ such that the sets $K$, $\phi_{1,n}(K),\dots,\phi_{N,n}(K)$ are pairwise disjoint.
\end{defn}
\begin{prop}
If $\phi_1,\dots,\phi_N$ are $N\geq2$ holomorphic self-maps of a domain $\Omega$ that are disjoint run-away, then the set
$$\{n\in \mathbb{N}:\ K, \ \phi_{1,n}(K),\dots,\phi_{N,n}(K)\ \textrm{are pairwise disjoint}\}$$ is infinite.
\end{prop}
\begin{proof}
Let $K\subset \Omega$ be a compact set,
then for $i,j\in \llbracket 0,N \rrbracket$ with $i\neq j$, there exists $n\in \mathbb{N}$ such that
$$\phi_{i,n}(K)\cap K= \emptyset \ \ \textrm{and} \ \ \phi_{i,n}(K)\cap \phi_{j,n}(K)=\emptyset.$$
Let $K_1:=K\cup\phi_{i,n}(K)\cup\phi_{j,n}(K)$ then $K_1$ is a non-empty compact subset of $\Omega$. Again because $\phi_1,\dots,\phi_N$ are disjoint run-away, there is an $m\in \mathbb{N}$ such that
$$\phi_{i,m}(K_1)\cap K_1=\emptyset \ \ \textrm{and} \ \ \phi_{i,m}(K_1)\cap\phi_{j,m}(K_1)=\emptyset.$$
Hence, if $m=n$, $$\phi_{i,2n}(K)\cap K=\emptyset \ \ \textrm{and} \ \ \phi_{i,2n}(K)\cap \phi_{j,2n}(K)=\emptyset,$$
if $m\neq n$, $$\phi_{i,m}(K)\cap K=\emptyset \ \ \textrm{and} \ \ \phi_{i,m}(K)\cap \phi_{j,m}(K)=\emptyset.$$ We can use this process infinitely many times.
\end{proof}
\begin{thm}
Let $\phi_1,\dots,\phi_N$ be $N\geq2$ holomorphic self-maps of $\mathbb{D}$. Then the following assertions
are equivalent:
\begin{enumerate}
\item The operators $C_{\phi_1},\dots,C_{\phi_N}$ are disjoint $\mathcal{F}$-transitive;
\item
\begin{enumerate}
\item each of $\phi_1,\dots,\phi_N$ is injective and
\item for every compact subset $K$ of $\mathbb{D}$ we have that $$\{n\in \mathbb{N}:\ K,\phi_{1,n}(K),\dots,\phi_{N,n}(K)\textrm{ are pairwise disjoint}\}\in \mathcal{F}.$$
\end{enumerate}
\end{enumerate}
\end{thm}
\begin{proof}
$(1)\Longrightarrow(2)$: Assume that $C_{\phi_1},\dots,C_{\phi_N}$ are disjoint $\mathcal{F}$-transitive. For $(a)$, since every operator $C_{\phi_i}$, $i=1,\dots,N$ is hypercyclic, this implies that the maps $\phi_1,\dots,\phi_N$ are injective. For $(b)$, let $K\subset \mathbb{D}$ be a compact and $0<M<1$ such that $K\subset B(0,M):=\{z\in \mathbb{C}:\ |z|<M\}$. Put

$$U_0:=\{f\in H(\mathbb{D}):\ |f(z)|<M,\ \forall z\in B(0,M)\}$$ and let $$U_i:=\{f\in H(\mathbb{D}):\ \frac{2i-1}{M}<|f(z)|<\frac{2i}{M},\ \forall z\in B(0,M)\}$$ for $i=1,\dots,N$. The sets $U_i$ are not empty for every $i=0,1,\dots,N$, since the constant functions satisfying the condition contain in $U_i(i = 0,1,\dots,N)$. Moreover, they are open, indeed, $U_0$ is clearly open and if $f \in U_i$, then the disc $B(f,\delta) \subset
 U_i$, where
 $$\delta = \min\{\frac{2i}{M}-\sup_{z\in K}|f(z)|,\sup_{z\in K}|f(z)|-\frac{2i-1}{M}\}.$$
Since $C_{\phi_1},\dots,C_{\phi_N}$ are disjoint $\mathcal{F}$-transitive, we have that
$$N(U_0,\dots,U_N):=\{n\in \mathbb{N}:\ U_0\cap C_{\phi_1}^{-n} U_1\cap C_{\phi_2}^{-n} U_2\cap\dots\cap C_{\phi_N}^{-n} U_N\neq \emptyset\}\in \mathcal{F}.$$
In order to prove $(2)$, we shall prove that
$$N(U_0,\dots,U_N)\subset \{n\in \mathbb{N}:\ K,\phi_{1,n}(K),\dots,\phi_{N,n}(K)\textrm{ are pairwise disjoint}\}.$$
Let $n_0\in \mathbb{N}$ such that $U_0\cap C_{\phi_{1}}^{-n_0} U_1\cap C_{\phi_2}^{-n_0} U_2\cap\dots\cap C_{\phi_N}^{-n_0} U_N\neq\emptyset$, then there is $f\in H(\mathbb{D})$ such that $f\in U_0$ and $C_{\phi_{i}}^{n_0}f\in U_i$ for every $i=1,\dots,N$. Hence $$\sup_{z\in B(0,M)}|f(z))|\leq M<\frac{1}{M}\leq \sup_{z\in B(0,M)}|f(\phi_{1,n_0}(z))|<\dots<
\sup_{z\in B(0,M)}|f(\phi_{N,n_0}(z))|,$$
and
$$\sup_{z\in B(0,M)}|f(\phi_{i,n_0}(z))|\leq\frac{2i}{M}<\frac{2j-1}{M}\leq\inf_{z\in B(0,M)}|f(\phi_{j,n_0}(z))|$$ for every $i< j=1,\dots,N,$
Thus, $\phi_{i,n_0}B(0,M)\cap B(0,M)=\emptyset$ and $\phi_{i,n_0}(B(0,M))\cap \phi_{j,n_0}(B(0,M))=\emptyset$ for every $i\neq j=1,\dots,N$.
Hence,  $\phi_{i,n_0}(K)\cap K=\emptyset$ and $\phi_{i,n_0}(K)\cap \phi_{j,n_0}(K)=\emptyset$.
Thus,
$$n_0\in\{n\in \mathbb{N} \mbox{ : } K,\phi_{1,n}(K),\dots,\phi_{N,n}(K)\ \textrm{are pairewise disjoint}\}.$$
Since $N(U_0,\dots,U_N)\in \mathcal{F}$ because $C_\phi$ is supposed to be disjoint $\mathcal{F}$-transitive, we have $\{n\in \mathbb{N}:$ $K,\phi_{1,n}(K),\dots,\phi_{N,n}(K)\ \textrm{are pairewise disjoint}\}\in \mathcal{F}$.\\
$(2)\Longrightarrow(1)$: Suppose that for every compact $K \subset \mathbb{D}$, the set
$$\{n \in \mathbb{N}: K,\phi_{1,n}(K),\dots,\phi_{N,n}(K)\textrm{ are pairwise disjoint}\}$$
belongs to $\mathcal{F}$. Let $U_0,\dots,U_N$ be non-empty open subsets of
$H(\mathbb{D})$. Then there exist $\varepsilon > 0$, a compact $K \subset \mathbb{D}$, and functions $g_0,\dots,g_N
\in H(\mathbb{D})$ such that
$$U_0 \supset\{h \in H(\mathbb{D}): \|h - g_0\|_K < \varepsilon\}$$ $$U_1 \supset\{h \in H(\mathbb{D}): \|h - g_1\|_K < \varepsilon\}$$ $$\vdots$$ and $$U_N \supset\{h \in H(\mathbb{D}): \|h - g_N\|_K < \varepsilon\}.$$
Let $L$ be a simply connected compact containing $K$, and let $n \in \mathbb{N}$ such
that $L,\phi_{1,n},\dots,\phi_{N,n}(L)$ are pairwise disjoint. Then the function $g_0$ is holomorphic on some neighborhood of $L$ and each $g_l \circ (\phi_{l,n})^{-1}$ is holomorphic on some
neighborhood of $\phi_{l,n}(L)$, $l=1,\dots,N$.
It results from Runge's approximation theorem that
there exists a function $h \in H(\mathbb{D})$ such that
$$\sup_{z\in L}
|g_0(z) - h(z)| < \varepsilon \ \ \textrm{and} \ \ \sup_{z\in \phi_{l,n}(L)}
|g_l\circ(\phi_{l,n})^{-1}(z) - h(z)| < \varepsilon.$$
Hence,
$$\sup_{z\in L}|g_0(z) - h(z)| < \varepsilon \ \ \textrm{and} \ \ \sup_{z\in L}|g_l(z) - h(\phi_{l,n}(z))| < \varepsilon,$$
which implies that
$$\sup_{z\in K}|g_0(z) - h(z)| < \varepsilon \ \ \textrm{and} \ \ \sup_{z\in K}|g_l(z) - h(\phi_{l,n}(z))| < \varepsilon.$$
Thus, $h\in U_0$ and $C_{\phi_{l}}^n(h) \in U_l$ for $l=1,\dots,N$.
Hence, $h\in U_0\cap C_{\phi_{1}}^{-n}U_1\cap\dots\cap C_{\phi_{N}}^{-n}U_N$.
This implies that
$$U_0\cap C_{\phi_{1}}^{-n}U_1\cap\dots\cap C_{\phi_{N}}^{-n}U_N\neq\emptyset.$$
Consequently,
$$\{n \in \mathbb{N}:\ L,\phi_{1,n}(L),\dots,\phi_{N,n}(L)\textrm{ are pairwise disjoint}\} \subset N(U_0,\dots,U_N),$$
and since $\{n \in \mathbb{N}:\ L,\phi_{1,n}(L),\dots,\phi_{N,n}(L)\textrm{ are pairwise disjoint}\} \in \mathcal{F},$
we have that $N(U_0,\dots,U_N) \in \mathcal{F}$. This shows that $C_{\phi_{1}},\dots,C_{\phi_N}$ are disjoint $\mathcal{F}$-transitive.
\end{proof}
\section{Disjoint $\mathcal{F}$-transitivity on $H(\Omega)$, when $\Omega$ is not simply connected}
In this section, we characterize the disjoint $\mathcal{F}$-transitivity of a sequence of composition operators on $H(\Omega)$, where $\Omega$ is not simply connected. In case $\Omega$ is finitely connected not simply connected we know that there are no disjoint transitivity of composition operator. Thus, it still to study the case $\Omega$ is infinitely connected.
\begin{thm}\label{t1}
Let $\Omega\subset \mathbb{C}$ be a domain of infinite connectivity, and let
$\phi_1,\dots,\phi_{N}$ be $N$ injective holomorphic self-map of $\Omega$. Then the following assertions
are equivalent:
\begin{enumerate}
\item The operator $C_{\phi_1},\dots,C_{\phi_N}$ are disjoint $\mathcal{F}$-transitive;
\item For every $\Omega$-convex compact subset $K$ of $\Omega$, the set
$\{n \in \mathbb{N}:$ $\phi_{i,n}(K)$ for $i=1,\dots,N$ are $\Omega$-convex and $K,\phi_{1,n}(K),\dots,\phi_{N,n}(K)$ are pairwise disjoint$\}$ belongs to $\mathcal{F}$.
\end{enumerate}
\end{thm}
To prove this theorem, we need the following theorem and lemma.
\begin{thm}\label{thm22}
Let $\Omega\subset \mathbb{C}$ be a domain, and $\phi_1,\dots,\phi_N$ be $N$ injective holomorphic self-maps of $\Omega$. Assume that $C_{\phi_1},\dots,C_{\phi_N}$ are disjoint $\mathcal{F}$-transitive. Then, for every $\Omega$-convex
compact subset $K$ of $\Omega$, the set
\small{
$$\{n \in \mathbb{N}:\ \phi_{1,n}(K),\dots,\phi_{N,n}(K)\textrm{ are }\Omega-convex\
\textrm{and}\ K,\phi_{1,n}(K),\dots,\phi_{N,n}(K)\textrm{ are pairwise disjoint}\}$$
}
belongs to $\mathcal{F}$.
\end{thm}
\begin{proof}
Let $K$ be a $\Omega$-convex compact subset of $\Omega$. Again by the proof of \cite[Theorem 3.12]{C2} we will be able to construct
a compact $L$ that contains $K$ such that:
\begin{enumerate}
\item $L$ is $\Omega$-convex and has a finite number of holes $O_1,\ldots,O_p$ whose
boundaries are Jordan curves $\gamma_1,\ldots,\gamma_p$;
\item $L$ is connected and its outer boundary $\gamma_0$ surrounds all the $p$ holes.
\end{enumerate}

%Let $M > 1$ such that $K \subset \overline{B}(z_0,M)$ for some $z_0 \in K$ and $b\in \Omega \backslash K$ such
%that $|z_0 - b| \geq 2M + 1$. Since $L$ is $\Omega$-convex, for every $j= 1, . . . , p$ we can
%find $a_j \in O_j \cap \mathbb{C}\backslash \Omega$. By increasing the distance from $b$ to $K$ if necessary, we
%can suppose that $|z - a_j| \leq |z - b|$ for every $z \in \overline{B}(z_0,M)$.

Let
$$g(z)=\frac{(z-b)^{p+1}}{\prod_{j=1}^{p}(z-a_j)},$$ for every $z\in \Omega$ and put $g_m(z)=mg(z)$ for $m\neq0$.
Then for every $j=1,\dots,p$ we have that
$$\frac{1}{2i\pi}\int_{\gamma_j}\frac{g'(z)}{g(z)}dz=1\ \ \mbox{ and }\ \ \frac{1}{2i\pi}\int_{\gamma_0}\frac{g'(z)}{g(z)}dz=1.$$
Choose two positive constant $M$ and $m$ such that for every $z\in K$, $m<|g(z)|<M.$
We put
$$U_0 =\{f \in H(\Omega):\ |f(z)| >2M\textrm{ for every }z\in K\ \textrm{and}\ f(z)\neq 0\textrm{ for every }z \in \Omega\backslash\{b\}\},$$
and
{\small $$U_i=\{f \in H(\Omega):\ r^{i-1}m<|f(z)|<r^{i-1}M\textrm{ for every }z\in K\ \textrm{and}\ f(z)\neq 0\textrm{ for every }z \in \Omega\backslash\{b\}\},$$}

for $ i=1,\dots,N,$
where $r$ is a constant taken to satisfy $0<r<\frac{m}{M}$. Clearly, the sets $U_i$ are non-empty open subsets of $H(\Omega)$ and $g_{r^{i-1}}\in U_i$, for every $i=1,\dots,N$.
Let $0 < \varepsilon < 1$ and put
$$W_0= \left\lbrace f \in H(\Omega): f(z) \neq 0\ \ \mbox{ and }\ \  \left\lvert\frac{f'(z)}{f(z)} - \frac{g'(z)}{g(z)}\right\rvert < \varepsilon, \forall z \in \Omega\backslash\{b\}\right\rbrace ,$$
for every $m\neq 0$, $g_m\in W_0$. Therefore, $W_0\cap U_i\neq\emptyset$ for $i=1,\dots,N$.

Let $n_0\in N(U_0,U_1\cap W_0,\dots,U_N\cap W_0)$, then
$$U_0\cap C_{\phi_1}^{-n_0}(U_1\cap W_0)\cap\dots\cap C_{\phi_N}^{-n_0}(U_N\cap W_0)\neq\emptyset,$$

so there is a holomorphic function $f\in U_0$, such that $f\circ\phi_{i,n_0}\in U_i$, for $i=1,\dots,N$.
Thus, we have that
$$\sup_{z\in K}|f(\phi_{N,n_0}(z))|\leq r^{N-1} M< r^{N-2} m$$ $$\leq \sup_{z\in K}|f(\phi_{N-1,n_0}(z))|\leq r^{N-2} M<\dots<m\leq\sup_{z\in K}|f(\phi_{1,n_0}(z))|\leq M<2M\leq\inf_{z\in K}|f(z)|$$
and
$$\sup_{z\in K}|f(\phi_{j,n_0}(z))|\leq r^{j-1}M<r^{j-1}m\leq r^{i-1}m\leq \inf_{z\in K}|f(\phi_{i,n_0}(z))|.$$
Hence, we obtain that the sets $K,\phi_{1,n_0}(K),\dots,\phi_{N,n_0}(K)$
are pairwise disjoint.

%and
%$$U_0 = \{f \in H(\Omega):\ |f(z)| > M \ \textrm{ for every } \ z \in B(z_0, M)
%\textrm{ and }f(z) \neq 0\textrm{ for every }z \in \Omega\backslash\{b\}\}.$$
%Let $0 < \varepsilon < 1$ and put
%$$W_0 := \{f \in H(\Omega): f(z) \neq 0\ \ \mbox{ and }\ \  \left\lvert\frac{f'(z)}{f(z)} - \frac{g'(z)}{g(z)}\right\rvert < \varepsilon, \forall z \in \Omega\backslash\{b\}\}.$$
%Then, $W_0 \cap U_0$ is non-empty open subset of $H(\Omega)$ (for instance, $g \in W_0 \cap U_0$). Let $U_1,\dots,U_N$ be open subsets of $V_0$, such that for every $z_i,z_j\in B(0,M)$, $\phi_i(z_i)\neq \phi_j(z_j)$, for $i\neq j$.
%Since $C_\phi$ is continuous, $C_{\phi_1}^{-s}U_1,C_{\phi_2}^{-s}U_2,\dots,C_{\phi_N}^{-s}U_N$ are non-empty open subsets of $H(\Omega)$.
%Also, $C_\phi$ is $\mathcal{F}$-transitive, then $N(W_0 \cap U_0, C_{\phi_1}^{-s} U_1,\dots,C_{\phi_N}^{-s}U_N) \in \mathcal{F}$. Let
%$m \in N(W_0 \cap U_0, C_\phi^{-s} U_1,\dots,C_\phi^{-s}U_N)$
%and put $n_0 = m - s$. Then $W_0 \cap U_0 \cap C_{\phi_{1}}^{n_0}U_1\cap\dots\cap C_{\phi_N}^{n_0}U_N$ is a non-empty subset of
%$H(\Omega)$. Thus, we can select a function $f_0 \in H(\Omega)$ such that $f_0 \in U_1\cap\dots\cap U_N$ and
%$C_{\phi_i}^{n_0}f_0 \in W_0 \cap U_0$ for every $i=1,\dots,N$. Hence, $\sup_{\overline{B}(z_0,M)}|f_0(z)| \leq 1/M < M \leq \inf_{\overline{B}(z_0,M)}
%|f_0 \circ \phi_{i,n_0}(z)|.$
%Then $\phi_{i,n_0}(K) \cap K = \emptyset$, $i=1,\dots,N$ and if $i\neq j$, $\phi_{i,n_0}(K)\cap \phi_{j,n_0}(K)=\emptyset$.

Now, since $f \in U_0$ and $f \circ \phi_{i,n_0} \in W_0 \cap U_i$, we have
that $f(z) \neq 0$ and $f \circ \phi_{i,n_0}(z) \neq 0$ for every $z\in \Omega\backslash\{b\}$, and for every $j =
0, 1, . . . , p,$ we have that
\begin{equation}\label{equ41}
\frac{1}{
2i\pi}\int_{\phi_{n_0}(\gamma_j)}\frac{f'(z)}{f(z)}dz =\frac{1}{2i\pi} \int_{\gamma_j}\frac{(f\circ \phi_{i,n_0})'(z)}{(f\circ \phi_{i,n_0})(z)}dz \geq \frac{1}{2i\pi}\int_{\gamma_j}\frac{g'(z)}{g(z)}dz - \varepsilon > 0.
\end{equation}

Presently, we will show that, for $i\in  \llbracket 1,N \rrbracket$, $\phi_{i,n_0}(L)$ is $\Omega$-convex. By \cite[Lemma 3.11]{C2} this implies that $\phi_{i,n_0}(K)$ is $\Omega$-convex, which will complete the proof.

Because $\phi_{i,n_0}$ is injective and $L$ has $p$ holes, the set $\phi_{i,n_0}(L)$ is a compact subset with exactly $p$ holes (see \cite[p. 276]{remmert}). Assume, by contradiction method,
that $\phi_{i,n_0}(L)$ is not $\Omega$-convex. Then one of these holes, that we denote by $O$,
is missing a point from $\mathbb{C}\backslash\Omega$. Since injective holomorphic functions
maps boundaries to boundaries, there is some $l \in \{0, 1, . . . , p\}$ such that
the Jordan curve $\phi_{i,n_0}(\gamma_l)$ is the boundary of $O$. Moreover, since $O$ contains
no point from $\mathbb{C}\backslash \Omega$, we have that
$$ind_{\phi_{i,n_0}}(\gamma_l)(\zeta) = 0,\ \textrm{for}\ \zeta\notin \Omega.$$
Now, the compact set $L$ is to the left of each curve $\gamma_j$, $j = 0, 1, . . . , p$.
Since injective holomorphic mappings preserve orientation, $\phi_{i,n_0}(L)$ must also
be to the left of the image curve $\phi_{i,n_0}(\gamma_l)$. This indicates that $\phi_{i,n_0}(\gamma_l)$ is negatively oriented. Since $f$ is holomorphic in a neighborhood of $O$, the integral
$$-\frac{1}{2i\pi}\int_{\phi_{i,n_0}(\gamma_l)}\frac{f'(z)}{f(z)}dz$$
equals the number of zeros of $f$ in $O$; but, by inequality \eqref{equ41}, that integral has a negative value, which is contradicting. Thus, we can conclude
that $\phi_{i,n_0}(L)$ is $\Omega$-convex. Consequently,
$n_0 \in \{n \in \mathbb{N}:$ $\phi_{1,n}(K),\dots,\phi_{N,n}(K)\textrm{ are }\Omega-convex\
\textrm{and}\ K,\phi_{1,n}(K),\dots,\phi_{N,n}(K)\textrm{ are pairwise disjoint}\}$.
Hence,
$N(U_0,W_0\cap U_1,\dots,W_0\cap U_N)\subset \{n\in \mathbb{N}:$ $\phi_n(K)$ is $\Omega$-convex and $K,\phi_{1,n}(K),\dots,\phi_{N,n}(K)$ are pairwise disjoint$\}$,
and the fact that $N(U_0,W_0\cap U_1,\dots,W_0\cap U_N) \in \mathcal{F}$ implies that
{\small
$$\{n \in \mathbb{N}:\ \phi_{1,n}(K),\dots,\phi_{N,n}(K)\textrm{ are }\Omega-convex\
\textrm{and}\ K,\phi_{1,n}(K),\dots,\phi_{N,n}(K)\textrm{ are pairwise disjoint}\}$$
}
belongs to $ \mathcal{F}.$

\normalsize
\end{proof}
\begin{lemma}\label{lem21}
Let $\Omega\subset \mathbb{C}$ be a domain, and let $\phi_1,\dots,\phi_N$ be $N$ injective holomorphic self-map of $\Omega$. Suppose that, for every $\Omega$-convex compact subset
$K$ of $\Omega$, the set
$$\{n \in \mathbb{N}:\ \phi_{i,n}(K)\textrm{ are }\Omega\textrm{-convex and }
K,\phi_{1,n}(K),\dots,\phi_{N,n}(K)\textrm{ are pairwise disjoint}\}$$
 is not empty. Then, for every connected $\Omega$-convex compact subset $K$ of $\Omega$ with more than two holes, the set
\small{
$$\{n\in \mathbb{N}\mbox{ : }K\cup \phi_{1,n}(K)\cup\dots \cup \phi_{N,n}(K) \mbox{ is }\Omega\mbox{-convex and }K, \ \phi_{1,n}(K),\dots,\phi_{N,n}(K)\mbox{ are pairwise disjoint}\}$$
}
contains the set
$$\{n\in \mathbb{N} \mbox{ : }\phi_{i,n}(K)\mbox{ are }\Omega\mbox{-convex and }K,\phi_{1,n}(K),\dots,\phi_{N,n}(K)
\mbox{ are pairwise disjoint}\}.$$
\end{lemma}
\begin{proof}
Let $K$ be a connected $\Omega$-convex compact subset with at least
two holes. Let $n_0 \in \{n\in \mathbb{N}:$ $\phi_{i,n}(K)$ are $\Omega$-convex and $K,\phi_{1,n}(K),\dots,\phi_{N,n}(K)$ are pairwise disjoint$\}$.
Then $\phi_{i,n_0}(K)$ is $\Omega$-convex and $K,\phi_{1,n_0}(K),\dots,\phi_{N,n_0}(K)$ pairwise disjoint. Since every hole of $K\cup \phi_{1,n_0}(K)\cup\dots \cup \phi_{N,n_0}(K)$ is a hole of every $\phi_{i,n_0}(K)$, $i=1,...,N$ which are $\Omega$-convex, then every hole of $K\cup \phi_{1,n_0}(K)\cup\dots \cup \phi_{N,n_0}(K)$ contains
a point of $\mathbb{C}\backslash\Omega$. Thus, $K\cup \phi_{1,n_0}(K)\cup\dots \cup \phi_{N,n_0}(K)$ is $\Omega$-convex. Hence, $n_0 \in \{n\in \mathbb{N}:$ $K\cup \phi_{1,n}(K)\cup\dots \cup \phi_{N,n}(K)$ is $\Omega$-convex and $K,\phi_{1,n}(K),\dots,\phi_{N,n}(K)$ are pairwise disjoint$\}$.
\end{proof}
Now we are ready to prove Theorem \eqref{t1}.
\begin{proof}[Proof of Theorem \ref{t1}]
The necessary implication was already shown in Theorem \eqref{thm22}. We will now prove the sufficient implication.
Assume that $(2)$ holds.
Let $U_0,\dots,U_N$ be non-empty open subsets of $H(\Omega)$. Then there exist
$\varepsilon > 0$, a compact subset $K\subset \Omega$, and $N+1$ functions $g_0,\dots,g_N\in H(\Omega)$ such that
$$U_0 \supset \{h \in H(\Omega):\ \|h - g_0\|_K < \varepsilon\}$$
$$\vdots$$
$$U_N\supset \{h \in H(\Omega):\ \|h - g_N\|_K < \varepsilon\}.$$
By enlarging $K$, we can assume that it is connected, $\Omega$-convex
and having at least two holes. Let $n_0 \in \{n\in \mathbb{N}:$ $\phi_{i,n}(K)$ are $\Omega$-convex and $K,\phi_{1,n}(K),\dots,\phi_{N,n}(K)$
are pairwise disjoint$\}$. By Lemma \eqref{lem21},
$K\cup \phi_{1,n_0}(K)\cup\dots\cup\phi_{N,n_0}(K)$ is $\Omega$-convex and $K,\phi_{1,n_0}(K),\dots,\phi_{N,n_0}(K)$ are pairwise disjoint. Then
the function $g_i \circ \phi_{i,n_0}^{-1}$ is holomorphic on $\phi_{i,n_0}(K)$ for every $i=1,\dots,N$, and $g_0$ is holomorphic on $K$.
Thus, by applying Runge's approximation theorem we can find a function $h \in H(\Omega)$
such that
$|h(z) - g_0(z)| < \varepsilon$ for $z \in K$, $|h(\zeta) - g_i(\phi_{i,n_0}^{-1}(\zeta))| < \varepsilon$ for $\zeta\in \phi_{i,n_0}(K)$ and $i=1,\dots,N$.
Hence,
$|h(\phi_{i,n_0}(z)) - g_i(z)| < \varepsilon$ for $z \in K$.
Thus, $h \in U_0$ and $C_{\phi_{i}}^{n_0}(h) \in U_i$ for $i = 1,\dots,N,$ which means that $U_0\cap C_{\phi_1}^{-n_0}U_1\cap\dots\cap C_{\phi_N}^{-n_0}U_N\neq\emptyset.$
This implies that $n_0 \in \{n \in \mathbb{N}:$ $U_0\cap C_{\phi_1}^{-n}U_1\cap\ldots\cap C_{\phi_N}^{-n}U_N\neq\emptyset\}$.
Consequently, the set
\small{
$$\{n\in \mathbb{N}\mbox{ : }\phi_{1,n}(K),\dots,\phi_{N,n}(K) \mbox{ are }\Omega\mbox{-
convex and }K,\phi_{1,n}(K),\dots,\phi_{N,n}(K)\mbox{ are pairwise disjoint}\}$$
}
includes in the set
$$\{n\in \mathbb{N}\mbox{ : } U_0\cap C_{\phi_1}^{-n}U_1\cap\dots\cap C_{\phi_N}^{-n}U_N\neq\emptyset\}$$
and thus by hereditary upward property of $\mathcal{F}$, we have that
$$\{n\in \mathbb{N}\mbox{ : }U_0\cap C_{\phi_1}^{-n}U_1\cap\dots\cap C_{\phi_N}^{-n}U_N\neq\emptyset\} \in \mathcal{F}.$$
Hence, $C_{\phi_1},\dots,C_{\phi_N}$ are disjoint $\mathcal{F}$-transitive on $H(\Omega)$.
\end{proof}
%%%%%%%%%%%%%%%%%%%%%%%%%%%%%%%%%%%%%%%%%%%%%%%%%%%%%%%%%%%%%%%%%%%%%%%%%%%%%%%%%%%%%%%%%%%%%%%%%%%%%%%%
\noindent\textbf{Acknowledgment}. The authors are sincerely grateful to the anonymous referees for their careful reading, critical comments and valuable suggestions that contribute significantly to improving the manuscript during the revision.
%%%%%%%%%%%%%%%%%%%%%%%%%%%%%%%%%%%%%%%%%%%%%%%%%%%%%%%%%%%%%%%%%%%%%%%%%%%%%%%%%%%%%%%%%%%%%%%%%%%%%%%%

\end{document}